\documentclass[12pt]{amsart}
\usepackage{preamble_2025-05-22}

\begin{document}

\title{Trianguline representations and locally analytic principal series of $\GL_2(\Qp)$}

\author{Matthias Strauch}
\address{Indiana University, Department of Mathematics, Rawles Hall, Bloomington, IN 47405, U.S.A.}
\email{mstrauch@iu.edu}

\author{Zichuan Wang}
\address{Indiana University, Department of Mathematics, Rawles Hall, Bloomington, IN 47405, U.S.A.}
\email{wangzich@iu.edu}

\begin{abstract}
Let $\rho$ be an absolutely irreducible 2-dimensional $p$-adic representation
of the absolute Galois group of $\Qp$, and let $\Pi(\rho)$ be the unitary Banach space representation of $G = \GL_2(\Qp)$ associated to $\rho$ by the $p$-adic Langlands correspondence. We deduce from results due to Colmez, Dospinescu, Pa\v{s}k\={u}nas, Emerton, and others that if the locally analytic representation $\Pi(\rho)^\la$ has a subquotient isomorphic to a non-zero subquotient of a locally analytic principal series representation of $G$, then $\rho$ is trianguline.
\end{abstract}

\maketitle

\tableofcontents

\section{Introduction}
Let $E/\Qp$ be a finite extension and $\rho: \sG_\Qp \ra \GL_2(E)$ an absolutely irreducible 2-dimensional continuous representation of the absolute Galois group $\sG_\Qp$ of $\Qp$. Let $\Pi(\rho)$ be the unitary Banach space representation of $G = \GL_2(\Qp)$ attached to $\rho$ by the $p$-adic Langlands correspondence, cf. \cite{Colmez10,ColmezDospinescuPaskunas14}. Let $\Pi(\rho)^\la$ be the locally analytic representation associated to $\Pi(\rho)$. $\rho$ is called trianguline if its $(\vphi,\Gamma)$-module $D_\rig(\rho)$ over the Robba ring $\sR$ is an extension of $(\vphi,\Gamma)$-modules of rank 1, i.e., there is an exact sequence

\[0 \lra \sR(\delta_1) \lra D_\rig(\rho) \lra \sR(\delta_2) \lra 0 \;,\]

\vskip8pt

where $\delta_1, \delta_2: \Qpx \ra E^\x$ are continuous (hence locally analytic) characters, cf. \cite[0.3]{Colmez_LAvec} for the notation.\footnote{We use the term trianguline for what is called {\it split trianguline} in \cite{Nakamura_Classification, Berger_Trianguline}.} Let $T$ (resp. $B$) be the group of diagonal matrices (resp. upper triangular matrices) in  $G$. Given a locally analytic character $\chi: T \ra E^\x$, we denote its inflation to $B$ again by $\chi$ and call the induced representation 

\[\Ind^G_B(\chi) = \{ f: G \ra E \midc f \mbox{ is locally analytic}, \forall g \in G, b \in B: f(bg) = \chi(b)f(g) \,\} \;,\]

\vskip8pt

a {\it locally analytic principal series representation}. This is an admissible representation (in the sense of \cite{ST_inventiones}) of length at most three, all of whose irreducible subquotients\footnote{Throughout this paper, {\it subrepresentation} refers to a closed subrepresentation, and a {\it subquotient} is the quotient of a closed subrepresentation by a closed subrepresentation. A representation is called irreducible if it is non-zero and has no non-zero proper closed subrepresentation. Similarly, the (Jordan-H\"older) length of a representation refers to the length of a filtration by closed subrepresentations with topologically irreducible quotients. This is a well-defined concept independent of the choice of such a filtration.} are well understood, cf. \ref{JH-princseries}.  We call an irreducible subquotient of a locally analytic principal series representation of $G$ a representation of {\it principal series type}. 

\vskip8pt

The ``if'' part of the following statement is well known, cf. \cite{Colmez_LAvec,LiuXieZhang,Liu_l-a-vectors-crystabelian}. Similarly, the ``only if'' part seems to be well known to experts, cf. \cite[sec. 3.2.6]{EmertonGeeHellmann}, but has apparently not been completely documented, which is the rationale for this note. 

\vskip8pt

{\bf Theorem.} {\it Let $\rho$ be as above. Then $\Pi(\rho)^\la$ has a subquotient which is of principal series type if and only if $\rho$ is trianguline.}

\vskip8pt

We sketch here the arguments for the ``only if'' direction. Absolutely irreducible non-trianguline representations $\rho$ can be divided into two classes:

\vskip8pt

(1) When no twist of $\rho$ is de Rham with distinct Hodge-Tate weights.

\vskip8pt

(2) When there is a continuous character $\vep: \sG_\Qp \ra E^\x$ such that $\vep \ot \rho$ is de Rham with distinct Hodge-Tate weights.

\vskip8pt

In \cite[Cor. 0.4]{Colmez_poids}, Colmez determines the Jordan-H\"older length of $\Pi(\rho)^\la$ for any $\rho$. For non-trianguline $\rho$ we have: if $\rho$ is of type (1), then $\Pi(\rho)^\la$ is irreducible, and if $\rho$ is of type (2), then $\Pi(\rho)^\la$ has length two. 

\vskip8pt

{\it Representations of type (1).} Suppose that $\rho$ is of type (1) and, by way of contradiction, assume $\Pi(\rho)^\la$ to be of principal series type. This representation is then isomorphic to an irreducible locally analytic principal series $B^\an(\delta_1,\delta_2) = \Ind^G_B(\delta_2 \ot \delta_1 \chi_\cyc^{-1})$, where $\chi_\cyc(x) = x|x|$, using the notation of \cite{Colmez_LAvec}. Therefore, this locally analytic representation embeds into $\Pi(\rho)$, and can thus be equipped with a $G$-invariant norm. By a result of Emerton \cite[0.3]{EmertonLfunctions}, one has $|\delta_1(p)| \le 1$. 

\vskip8pt

If $|\delta_1(p)| = 1$, then the inducing character is unitary and the map $B^\an(\delta_1,\delta_2) \hra B^\cont(\delta_1,\delta_2)$ into the continuous induction $B^\cont(\delta_1,\delta_2)$ extends to a map $\alpha: \Pi(\rho) \ra B^\cont(\delta_1,\delta_2)$, because $\Pi(\rho)$ is the universal unitary completion of $\Pi(\rho)^\la \simeq B^\an(\delta_1,\delta_2)$. $\alpha$ is easily seen to be an isomorphism. However, this is impossible as the image of the functor $\Pi$ does not contain so-called ordinary unitary Banach space representations \cite[1.1]{ColmezDospinescuPaskunas14}, \cite[1.3]{Paskunas_Montreal}. 

\vskip8pt

If $|\delta_1(p)| < 1$, then $s = (\delta_1,\delta_2,\infty)$ is a point on the trianguline variety $\sS_\irr$, and $B^\an(\delta_1,\delta_2)$ embeds into $\Pi(V(s))$, where $V(s)$ is the Galois representation attached to $s$. This embedding extends to a map $\Pi(\rho) \ra \Pi(V(s))$ which can be shown to be an isomorphism. Hence $\rho \simeq V(s)$ is trianguline, contradicting our assumption.

\vskip8pt

{\it Representations of type (2).} We may assume without loss of generality that $\rho$ is de Rham, and we suppose in this introduction that the associated Weil group representation $\WD(D_\pst(\rho))$ is absolutely irreducible.\footnote{See case (2b) in the proof of \ref{main-result} for the case of irreducible but not absolutely irreducible $\WD(D_\pst(\rho))$.} Using the change of weights arguments of \cite{Colmez_poids}, one can reduce to the case when $\rho$ is de Rham and has Hodge-Tate weights 0 and 1. In this situation Dospinescu and Le Bras have shown in \cite[1.4]{DospinescuLeBras} that $\Pi(\rho)^\la$ sits in an exact sequence 

\[0 \lra \pi := \LL\left(\WD(D_\pst(\rho))\right) \lra \Pi(\rho)^\la \lra \left(H^0(\Sigma_n,\cO)^\psi\right)'_b \lra 0\]

\vskip8pt

where $\LL$ denotes the local Langlands correspondence\footnote{Associating smooth representations to 2-dimensional Weil-Deligne representations, normalized as in \cite[VI.6 \S 11, p. 465] {Colmez10}.}, $\Sigma_n$ is a Drinfeld covering space of the $p$-adic upper half plane, and $\psi = {\rm JL}(\pi)$ is the representation of $D^\x$ corresponding to $\pi$ by the local Jacquet-Langlands correspondence, where $D$ is the quaternion division algebra over $\Qp$. The notation $(-)'_b$ on the right denotes the continuous dual space (with the strong topology). As $\pi$ is not of principal series type, it thus remains to show that the representation on the right is not of principal series type. For this it suffices to show that the naive Jacquet module of that representation vanishes. Equivalently, we must show that $H^0(\Sigma_n,\cO)^\psi$ has vanishing $U$-invariants, where $U \sub G$ is the upper triangular unipotent subgroup. It is not hard to see that $U$-invariant rigid analytic functions on $\Sigma_n$ are locally constant. The action of $D^\x$ on the (geometrically) connected components is known to factor through the norm $D^\x \ra \Qpx$. As $\psi$ is not one-dimensional, the space of $U$-invariant vectors in $H^0(\Sigma_n,\cO)^\psi$ is zero.

\vskip8pt

{\it Further remarks.} In \cite[Appendix,Remark 2]{ST_UgFinite} the question had been raised whether all absolutely irreducible admissible locally analytic representations are subquotients of locally analytic parabolic inductions of locally finite-dimensional representations of Levi subgroups.\footnote{The condition imposed in loc.cit. is that the strong dual is topologically simple and analytic, but this is equivalent to it being topologically simple and coadmissible. Moreover, in this remark we only consider absolutely irreducible representations in order to avoid rationality issues.} For $G = \GL_2(\Qp)$ this would mean that every absolutely irreducible admissible locally analytic representation is (a) of the form
$(\lambda \circ \det) \ot \Sym^k \ot \pi$ with a locally analytic character $\lambda$ and a smooth absolutely irreducible representation $\pi$, or is (b) an irreducible locally analytic principal series representation. As we show here, this is not the case for $\Pi(\rho)^\la$ when $\rho$ is of type (1) and stays non-trianguline after any base change. Moreover, for de Rham representations $\rho$ of type (2) with absolutely irreducible $\WD(D_\pst(\rho))$ the quotient $\Pi(\rho)^\la/\Pi(\rho)^\lalg$ is not a representation of the form (a) or (b). For such $\rho$ the representation $\Pi(\rho)^\la$ is nevertheless somewhat accessible as it has a geometric description in terms of the Drinfeld tower. For non-trianguline $\rho$ which are not twists of de Rham representations, the representations $\Pi(\rho)^\la$ have not yet been explicitly described. 

\vskip8pt

{\it Acknowledgments.} We thank Zijian Yao for raising the question of whether the result discussed here had already been completely documented in the existing literature.

\section{Jordan-H\"older factors of some locally analytic representations of \texorpdfstring{$\GL_2(\Qp)$}{}}

Throughout this paper $E/\Qp$ denotes a finite field extension. The purpose of this section is to gather some facts for later reference. The results in sections \ref{principal-series} - \ref{Colmez} are already known and are restated here for convenient reference in the proof of \cref{main-result}.

\subsection{Locally analytic principal series representations of \texorpdfstring{$\GL_2(F)$}{}}\label{principal-series}

In this subsection only, we let $F$ be an intermediate field $\Qp \sub F \sub E$, and consider locally $F$-analytic representations of $\GL_2(F)$ on $E$-vector spaces.

\vskip8pt

\begin{prop}\label{JH-princseries} Let $W$ be an irreducible subquotient of a locally $F$-analytic principal series representation of $G = \GL_2(F)$ over $E$. Then $W$ is isomorphic to a representation of exactly one of the following types:

\vskip8pt

\begin{enumerate}
\item  $(\lambda \circ \det) \otimes \Sym^k(E^2) \otimes \pi$, where $\lambda: \Fx \ra \Ex$ is a locally $F$-analytic character, and $\pi$ is an irreducible smooth representation of $G$. The latter is either the trivial one-dimensional representation, the Steinberg representation, or a smooth principal series representation.
        
\vskip5pt

\item An irreducible locally $F$-analytic principal series representation.
\end{enumerate}

\vskip8pt

Moreover, the underlying vector space of $W$ has countable (possibly finite) dimension over $E$ if and only if $W$ is of type (1).
\end{prop}

\begin{proof} Let $T \sub G$ be the maximal torus of diagonal matrices. Let $\chi: T \ra \Ex$ be a locally analytic character, and define $c(\chi) \in E$ by the equation 
$\chi(\diag(t,t^{-1})) = t^{-c(\chi)}$ for $t \in F$ sufficiently close to 1. If $c(\chi) \notin \bbZ_{\ge 0}$, then the locally analytic induction $\Ind^G_B(\chi)$, where $B$ is the Borel subgroup of upper triangular matrices in $G$, is topologically irreducible, cf. \cite[3.1.6]{KisinStrauch06}, \cite[4.2.2]{OrlikStrauch10}. Then we are in case (2) of the statement.

\vskip8pt

Now we assume $k := c(\chi) \in \Z_{\ge 0}$. Write $\chi(\diag(a,d)) = \chi_1(a)\chi_2(d) = \frac{\chi_1}{\chi_2}(a) \chi_2(ad)$ with locally analytic characters $\chi_1, \chi_2: \Fx \ra \Ex$. Then an easy computation shows that $k = -\left(\frac{\chi_1}{\chi_2}\right)'(1)$. Hence $\frac{\chi_1}{\chi_2}(a) = a^{-k} \tau(a)$ with a smooth character $\tau: \Fx \ra \Ex$. 

\vskip8pt

Set $\vep_1(\diag(a,d)) = a$, so that $\chi = (\chi_2 \c \det) \cdot \vep_1^{-k} \cdot (\tau \c \vep_1)$. It follows that $\Ind^G_B(\chi) = (\chi_2 \c \det) \ot \Ind^G_B(\vep_1^{-k} \cdot (\tau \c \vep_1))$. Moreover, we have 

\[\Ind^G_B(\vep_1^{-k} \cdot (\tau\c \vep_1)) = \cF^G_B(M(k\rmd \vep_1),\tau \c \vep_1) \;,\]

\vskip8pt

where $\cF^G_B(-,-)$ is the bi-functor introduced in \cite[4.6.2]{OrlikStrauchJH}, and $M(k\rmd \vep_1)$ is the Verma module associated to the dominant integral weight $k\rmd \vep_1$. This Verma module has the finite-dimensional irreducible representation $L(k\rmd \vep_1) \simeq \Sym^k(E^2)$ as a quotient, and we have an exact sequence

\[0 \lra M(k\rmd\vep_1 - (k+1)\rmd\alpha) \lra M(k\rmd \vep_1) \lra L(k\rmd \vep_1) \lra 0 \;.\]

Here the Verma module on the left is the submodule of $M(k\rmd \vep_1)$ generated by $y^{k+1}.v^+$, where $v^+$ is a generating vector of $M(k\rmd \vep_1)$ and $y = \begin{pmatrix}
    0 & 0 \\ 1 & 0 
\end{pmatrix}$, and $y^{k+1}.v^+$ has therefore weight $k\rmd \vep_1 - (k+1)\rmd \alpha$. The Verma module on the left is irreducible (already as a module over $\frs\frl_2$). By the exactness of $\cF^G_B$ \cite[4.9 (i)]{OrlikStrauchJH} one has an exact sequence 

\[0 \lra \cF^G_B(L(k\rmd \vep_1), \tau \c \vep_1) \lra \cF^G_B(M(k\rmd \vep_1), \tau \c \vep_1)  \lra \cF^G_B(M(k\rmd\vep_1 - (k+1)\rmd\alpha), \tau \c \vep_1) \lra   0 \;.\]

\vskip8pt

The representation on the right is irreducible by \cite[5.8]{OrlikStrauchJH}. Moreover, the ``$P$-$Q$ formula'' \cite[4.9 (ii)]{OrlikStrauchJH} gives 

\[\cF^G_B(L(k\rmd \vep_1), \tau \c \vep_1) \cong \cF^G_G\Big(L(k\rmd \vep_1), \ind^G_B(\tau \c \vep_1)\Big) \cong \Sym^k(E^2) \ot_E \ind^G_B(\tau \c \vep_1) \;,\]

\vskip8pt

where $\ind$ denotes smooth induction. Using \cite[5.8]{OrlikStrauchJH} again, we see that this representation is irreducible if the smooth representation $\ind^G_B(\tau \c \vep_1)$ is irreducible. If the latter is reducible, then it has a Jordan-H\"older filtration of length two whose subquotients are a one-dimensional and an infinite-dimensional representation isomorphic to a twist of the Steinberg representation $\St$. Applying once more the exactness of $\cF^G_G$ and \cite[5.8]{OrlikStrauchJH}, we conclude that when $\ind^G_B(\tau \c \vep_1)$ is reducible, the representation $\cF^G_B(L(k\rmd \vep_1), \tau \c \vep_1)$ has (topological) Jordan-H\"older length 2, and the two irreducible subquotients are of the form $(\tilde{\tau} \c \det) \ot_E \Sym^k(E^2)$ and $(\tilde{\tau} \c \det) \ot_E \St \ot_E \Sym^k(E^2)$, for a smooth character $\tilde{\tau}$.

\vskip8pt

Clearly, if $W$ is of type (1), it is of countable dimension. On the other hand, choosing a locally analytic section of the projection $G \ra B\bksl G$, one sees that any locally analytic principal series representation of $G$ over $E$ has an underlying vector space which is isomorphic to the space $C^\la(B\bksl G,E)$ of locally analytic functions on $B \bksl G \simeq \bbP^1(F)$, and this space is of uncountable dimension over $E$. 
\end{proof}

\vskip8pt

\begin{rem} The irreducibility and Jordan-H\"older series of locally analytic principal series representations of $\GL_2(\Qp)$ have also been discussed in \cite[Thm. 6.1 and Rem. 2 after Prop. 6.2]{ST_JAMS} (note that the definition of $c(\chi)$ given there differs by a sign from the one used in the preceding proof). 
\end{rem}

\vskip8pt

\subsection{Unitarizable principal series representations of \texorpdfstring{$\GL_2(\Qp)$}{}}\label{unitarizable}

Here, and for the remainder of this paper, we return to the case of $G = \GL_2(\Qp)$. Given a locally analytic character $\delta$ of $\Qpx$ we set $u(\delta) = v_p(\delta(p))$ and $w(\delta) = \delta'(1)$. Let $T \sub B \sub G$ be as in the introduction.

\vskip8pt

\begin{prop}\label{completion-Emerton} Let $\chi: T \ra E^\x$ be a locally analytic character, and write $\chi = \delta_2 \ot \delta_1 \chi_\cyc^{-1}$ so that $\Ind^G_B(\chi) = B^\an(\delta_1,\delta_2)$ in the notation of \cite[0.4]{Colmez_LAvec}. Assume that this representation is irreducible and has a non-zero unitary Banach space completion.\footnote{By this we mean that this representation contains a $G$-invariant $\cO_E$-lattice which does not contain a non-zero $E$-vector space.} Then:

\vskip8pt

\begin{enumerate}[(a)]
    \item $u(\delta_1) + u(\delta_2) = 0$;
    \item $u(\delta_1) \ge 0$;
    \item $w(\delta_1\delta_2^{-1}) \notin \bbZ_{>0}$.
\end{enumerate}
\end{prop}

\begin{proof} Note that the group element $g = \left(\begin{array}{cc} p & 0 \\0 & p \end{array}\right)$ acts on $V = B^\an(\delta_1,\delta_2)$ by multiplication with the scalar $(\delta_1\delta_2)(p)$. As we assume that $V$ has a non-zero unitary Banach space completion, the element $g$ must act by multiplication with a scalar in $\cO_E^\x$. Therefore, condition (a) holds. 

\vskip8pt

Now we show that condition (b) holds. Let $\ovB$ be the subgroup of lower triangular matrices.
By \cite[0.3]{EmertonJacII} there exists a closed $T$-equivariant embedding $\chi.\delta_\ovB \hookrightarrow J_\ovB(V)$, where $\delta_\ovB$ is the Haar modulus character of $\ovB$ and $J_\ovB$ is the Jacquet functor of \cite{Emerton_Jac-I}. Sibce $V$ is of compact type and topologically irreducible, \cite[1.6]{EmertonLfunctions}, applied to $P = \ovB$ and the character $\chi.\delta_\ovB$ of $T$, gives

\[\forall t \in T^+: \;|{\delta_\ovB}^{-1}(t) (\chi.\delta_\ovB)(t)| = |\chi(t)| \leq 1 \;,\] 

\vskip8pt

where $N^0_\ovB = \begin{pmatrix}1 & 0 \\ \mathbb{Z}_p & 1
\end{pmatrix}$ and $T^+ = \left\{t \in T \mid t N^0\ovB t^{-1} \subset N^0_\ovB\right\}$. For $t_0 = \begin{pmatrix}
1 & 0 \\ 0 & p
\end{pmatrix} \in T^+$ one has

\[|\chi(t_0)| = |\delta_2(1) \delta_1(p)(p|p|)^{-1}| = |\delta_1(p)| \leq 1 \;,\]

\vskip8pt

hence condition (b) holds.

\vskip8pt

(c) It is immediate to check that $w(\delta_1\delta_2^{-1})-1 = c(\chi)$, where $c(\chi)$ is defined as in the proof of \ref{JH-princseries}. Therefore, (c) is equivalent to $c(\chi) \notin \Z_{\ge 0}$, and this is precisely the case when $\Ind^G_B(\chi)$ is irreducible, by, e.g. \cite[3.1.6, 3.1.8]{KisinStrauch06} (or the discussion in the proof of \ref{JH-princseries}). 
\end{proof}

\vskip8pt

\subsection{Locally analytic vectors in Banach space representations of \texorpdfstring{$\GL_2(\Qp)$}{}}\label{Colmez}

\begin{thm}[{\cite{Colmez_poids}}]\label{Colmez-poids} Let $\rho: \sG_\Qp \ra \GL_2(E)$ be an absolutely irreducible continuous representation, and let $\Pi(\rho)$ be the unitary Banach space representation of $G = \GL_2(\Qp)$ associated to $\rho$ by the $p$-adic Langlands correspondence.

\vskip8pt

\begin{enumerate}
\item Suppose $\rho$ is neither trianguline nor a twist of a de Rham representation with distinct Hodge-Tate weights. Then $\Pi(\rho)^\la$ is irreducible.

\vskip5pt

\item Suppose $\rho$ is not trianguline and there is a continuous character $\vep$ of $\sG_\Qp$ such that $\vep \ot \rho$ is de Rham with distinct Hodge-Tate weights. Then $\Pi(\rho)^\la$ is of length two and the subspace of $\Pi(\rho)^{\SL_2(\Qp)\mbox{-}\lalg}$ of locally algebraic vectors for the action of $\SL_2(\Qp)$ is the only proper non-zero subrepresentation.

\vskip5pt

\item Suppose $\rho$ is trianguline. Then all irreducible subquotients of $\Pi(\rho)^\la$ are of principal series type. 
More precisely, if the $(\varphi, \Gamma)$-module $\Delta(\rho)$ attached to $\rho$ is an extension of $\mathcal R(\delta_2)$ by $\mathcal R(\delta_1)$ where $\delta_1, \delta_2$ are locally analytic characters of $\Q_p^\times$, then the semisimplification $\Pi(\rho)^{\textup{la,ss}}$ of $\Pi(\rho)^\la$ has the following description:
\[
\Pi(\rho)^{\textup{la,ss}} = \Ind^G_B( \delta_2 \ot \delta_1 \chi_\cyc^{-1})^{\textup{ss}} \oplus \Ind^G_B( \delta_1 \ot \delta_2 \chi_\cyc^{-1})^{\textup{ss}} 
\]
\end{enumerate}
    
\end{thm}

\begin{proof} (1) This holds by \cite[Thm. 0.3, (i)]{Colmez_poids}. 

\vskip8pt

(2) That $\Pi(\rho)^\la$ has length two in these cases is \cite[Cor. 0.4, (ii), first and second bullet points]{Colmez_poids}. As $\rho \mapsto \Pi(\rho)$ commutes with twisting by characters, we may assume that $\rho$ is de Rham with Hodge-Tate weights $a<b$. The assertion about the locally algebraic vectors then follows from \cite[Thm. 0.20]{Colmez10}.

\vskip8pt

(3) This is \cite[Rem. 0.2, (ii), first case]{Colmez_poids}.
\end{proof}

\section{Trianguline representations and representations of principal series type}
\addtocounter{subsection}{1}

We begin with a lemma which will be useful in the proof of \ref{main-result} below.

\vskip8pt

\begin{lemma}\label{dim-of-l-a-vectors} Let $\Pi$ be an admissible unitary $E$-linear Banach space representation of $G = \GL_2(\Qp)$. Denote by $\Pi^\SLlalg$ the closed subspace of $\Pi^\la$ vectors which are locally algebraic for the action of $\SL_2(\Qp)$. If $\Pi^\la/\Pi^\SLlalg \neq 0$, then $\Pi^\la$ is of uncountable dimension as a vector space over $E$.
\end{lemma}

\begin{proof} The representation $\Pi^\la$ is admissible by \cite[7.1]{ST_inventiones} whose notation and results we use in the following. Let $H \sub G$ be an open compact subgroup, and fix a strictly increasing sequence $(r_n)_n$ in $(\frac{1}{p},1) \cap p^\Q$ converging to $1$. Let $M = (\Pi^\la)'_b$ be the continuous dual space, equipped with the strong topology. Then $M_n := D_{r_n}(H,E) \ot_{D(H,E)} M$ is a finitely generated module over the noetherian Banach algebra $D_{r_n}(H,E)$. By \cite[16.5]{NFA} one has $\Pi^\la = \varinjlim_n (M_n)'_b$. Because the system $(M_n)_n$ is a coherent sheaf, the map $M \ra M_n$ has dense image, and the map $(M_n)'_b \ra \Pi^\la$ is thus injective. Hence, if $\Pi^\la$ is of countable dimension, then this is also the case for the Banach spaces $(M_n)'_b$. A Baire category argument shows that each $(M_n)'_b$ must be finite-dimensional. By \cite[3.2]{ST_UgFinite}, each $(M_n)'_b$ is contained in the locally algebraic vectors for the action of $\SL_2(\Qp)$. Since $\Pi^\la  = \varinjlim_n (M_n)'_b$, this would imply $\Pi^\la  = \Pi^\SLlalg$, contrary to the hypothesis.
\end{proof}

\begin{thm}\label{main-result} Let $E/\Qp$ be a finite extension and $\rho: \sG_\Qp \ra \GL_2(E)$ an absolutely irreducible continuous representation. Let $\Pi(\rho)$ be the unitary Banach space representation of $G = \GL_2(\Qp)$ associated to $\rho$ by the $p$-adic Langlands correspondence. Then the following assertions are equivalent:

\vskip8pt

\begin{enumerate}[(i)]
\item The locally analytic representation $\Pi(\rho)^\la$ has a subquotient which is isomorphic to a non-zero subquotient of a locally analytic principal series representation of $G$.

\vskip5pt

\item $\rho$ is trianguline.
\end{enumerate}  
\end{thm}

\begin{proof} ``(ii) $\Rightarrow$ (i)''. This is implied by part (3) of \ref{Colmez-poids}. In fact, every irreducible subquotient of $\Pi(\rho)^\la$ is of principal series type.

\vskip8pt

``(i) $\Rightarrow$ (ii)''. Suppose that $\rho$ is not trianguline. Then $\rho$ is exactly one of the types described in (1) and (2) of \ref{Colmez-poids}. Set $\Pi := \Pi(\rho)$.

\vskip8pt

(1) We assume that $\rho$ is of type (1) in \ref{Colmez-poids}. If $\rho$ is a twist of a de Rham representation with equal Hodge-Tate weights, then we may twist $\rho$ and henceforth assume without loss of generality that it is actually de Rham. Then $\Pi^\lalg = \Pi^\SLlalg$. By \ref{Colmez-poids}, $\Pi^\la$ is topologically irreducible. Suppose that $\Pi^\la$ is a subquotient of a locally analytic principal series representation. By \cite[0.20]{Colmez10} we have $\Pi^\lalg = 0$, hence $\Pi^\SLlalg = 0$, and by \ref{dim-of-l-a-vectors} $\Pi^\la$ is not of countable dimension over $E$. By \ref{JH-princseries}, it follows that $\Pi^\la$ is itself isomorphic to a principal series representation $\Ind^G_B(\chi)$. Hence $\Ind^G_B(\chi)$ embeds into the Banach space representation $\Pi$, and can thus be equipped with a $G$-invariant norm. Write $\chi = \delta_2 \ot \delta_1 \chi_\cyc^{-1}$, where $\chi_\cyc(x) = x|x|$ is the $p$-adic cyclotomic character. We then have $|(\delta_1\delta_2)(p)| = |\chi(\diag(p,p))| = 1$ because $\Ind^G_B(\chi)$ embeds into a unitary Banach space representation.

\vskip8pt

{\it Case $u(\delta_1)=0$.} Then $\chi$ is unitary, and the locally analytic induction $\Ind^G_B(\chi)$ embeds into the continuous induction $\Ind^G_B(\chi)^\cont$. By \cite[VII.11]{ColmezDospinescu14}, $\Pi(\rho)$ is the universal completion of $\Pi^\la$, and hence there is a continuous $G$-homomorphism $\alpha: \Pi \ra \Ind^G_B(\chi)^\cont$. This has dense image and is strict by \cite[6.2.9]{EmertonA}, and thus has closed image. The map is therefore surjective. Since $\Pi$ is irreducible, it is also injective and hence an isomorphism of Banach space representations. Because $\Ind^G_B(\chi)^\cont$ is an ordinary Banach space representation, this contradicts \cite[1.1]{ColmezDospinescuPaskunas14}, \cite[1.3]{Paskunas_Montreal}.

\vskip8pt

{\it Case $u(\delta_1) > 0$.} By \ref{completion-Emerton} the triple $s = (\delta_1,\delta_2,\infty)$ is a point of the subset $\sS^{\rm ng}_\ast$ of the trianguline variety $\sS_\irr$ of Colmez, see, e.g., \cite[0.3]{Colmez_LAvec}. By \cite[8.7 (i)]{Colmez_LAvec}, the locally analytic representation $\Ind^G_B(\chi) = B^\an(\delta_1,\delta_2)$ embeds into the unitary Banach representation $\Pi(V(s))$. 

\vskip8pt

On the other hand, using again that the universal completion of $\Pi^\la$ is equal to $\Pi$  \cite[VII.11]{ColmezDospinescu14}, we deduce that the embedding $\Ind^G_B(\chi) \hra \Pi(V(s))$ factors through $\Pi$, and we thus have a non-zero continuous $G$-homomorphism $\alpha: \Pi \ra \Pi(V(s))$. Moreover, $\alpha$ has dense image because the image contains $\Ind^G_B(\chi)$ and $\Pi(V(s))$ is irreducible. By the same arguments as above, we conclude that $\alpha$ must be an isomorphism. By \cite[1.1]{ColmezDospinescuPaskunas14}, this implies $\rho \simeq V(s)$, which contradicts our assumption that $\rho$ is not trianguline.

\vskip8pt

(2) Here we consider a non-trianguline representation $\rho$ of  type (2) in \ref{Colmez-poids}. As the $p$-adic local Langlands correspondence is compatible with twisting by characters \cite[III.13]{ColmezDospinescu14}, we can assume without loss of generality that $\rho$ is de Rham with Hodge-Tate weights $0$ and $k>0$. Again we have $\Pi^\lalg = \Pi^\SLlalg$. We distinguish two cases.

\vskip8pt

(2a) Suppose the associated Weil group representation $\WD(D_\pst(\rho))$ is absolutely irreducible. 

\vskip8pt

{\it Step 1.} We now use \cite{Colmez_poids} whose notation we will largely follow. Let $M$ be the Deligne-Fontaine module associated to $\rho$, and let $\sL$ be the 1-dimensional $E$-vector subspace of $M_\dR = (\Qpb \ot_{\Qpnr} M)^{\sG_\Qp}$ corresponding to the filtration associated to $\rho$. Let $\Delta_{k,\sL}$ be the $(\vphi,\Gamma)$-module corresponding to $\rho$ (which is determined by $M$, $k$, and $\sL$), and let $\Delta$ be the $(\vphi,\Gamma)$-module corresponding to the representation of $\sG_\Qp$ of Hodge-Tate weights $(0,0)$ with Deligne-Fontaine module $M$. By \cite[0.6 (iii)]{Colmez_poids} there is an admissible locally analytic representation $\Pi(M,-k)$ of $G$ which fits in an exact sequence

\[0 \lra \Pi(\Delta_{k,\sL})^\lalg \lra \Pi(\Delta_{k,\sL})^\la \lra \Pi(M,-k) \ot {\det}^k \lra 0 \;,\]

\vskip8pt

where $\Pi(\Delta_{k,\sL})^\la \simeq \Pi^\la$ and the superscript ``lalg'' denotes locally algebraic vectors (Colmez writes ``$x^k$'' for the twist by $\det^k$ on the right-hand side of the exact sequence). In particular, $\Pi(\Delta_{k,\sL})^\lalg = \Pi^\lalg$. 

\vskip8pt

{\it Step 2.} By \cite[0.20, 0.21]{Colmez10} we have $\Pi^\lalg = \Sym^{k-1}(E^2) \ot_E \pi$, which is irreducible by \ref{Colmez-poids} (2). Moreover, $\pi = \LL(M)$ is smooth and supercuspidal. Because $\pi$ is of countable dimension, if $\Pi(\rho)^\lalg$ were isomorphic to an irreducible subquotient of a principal series representation, it would have to be of the form $(\lambda \circ \det) \ot_E \Sym^{\wtk}(E^2) \ot_E \wtpi$, with a smooth irreducible representation $\wtpi$ which is not supercuspidal, by \ref{JH-princseries}. Viewing these as Lie algebra representations shows that $\wtk = k-1$ and that $\lambda$ is smooth. After replacing $(\lambda \c \det) \ot_E \wtpi$ by $\wtpi$ we may therefore assume $\lambda = \triv$.

\vskip8pt

But if $\Sym^{k-1}(E^2) \ot_E \pi \simeq \Sym^{k-1}(E^2) \ot_E \wtpi$, then 

\[\pi \simeq \varinjlim_K \Hom_K(\Sym^{k-1}(E^2),\Pi(\rho)^\lalg) \simeq \varinjlim_K \Hom_K(\Sym^{k-1}(E^2),\Sym^{k-1}(E^2) \ot_E \wtpi) \simeq \wtpi \;,\]

\vskip8pt

where the inductive limit is over the open subgroups $K \sub G$, and the isomorphisms on the left and right hold by \cite[App., Thm. 1]{ST_UgFinite}, \cite[4.2.8]{EmertonA} (and the arguments given in the proofs of these results). This contradicts the supercuspidality of $\pi$. 

\vskip8pt

{\it Step 3.} We now consider $V := \Pi^\la/\Pi^\lalg = \Pi^\la/\Pi^\SLlalg$, which is irreducible by part (2) of \ref{Colmez-poids}. Our goal is to show that $V$ is not isomorphic to an irreducible subquotient of a locally analytic principal series representation. Suppose, to the contrary, that $V$ were isomorphic to such a representation. By \Cref{dim-of-l-a-vectors}, $V$ is not countable-dimensional; \ref{JH-princseries} would therefore imply that $V$ itself is an irreducible principal series representation. From the exact sequence

\[0 \lra \Pi(\Delta_{k,\sL})^\lalg \lra \Pi(\Delta_{k,\sL})^\la \lra \Pi(M,-k) \ot {\det}^k \lra 0 \;,\]

\vskip8pt

we deduce that $V \cong \Pi(M,-k) \ot {\det}^k$. Let $U$ be the unipotent radical of $B$. 
By \cite[3.18, 3.20]{Colmez_poids}, we have the following isomorphism as $U$-representations:

\[\Pi(M,-k) \ot {\det}^k \cong \Pi(M,-1) \ot {\det} \,.\] 

\vskip8pt

The {\it naive Jacquet module} (with respect to the Borel $B$) $J^{\rm naive}(Z)$ of a locally analytic representation $Z$ is defined to be the maximal Hausdorff quotient of $Z$ on which $U$ acts trivially. It follows that  $J^{\rm naive}(V)$ vanishes if and only if $J^{\rm naive}(\Pi(M,-1) \ot {\det})$ vanishes.
It suffices to show that $J^{\rm naive}(\Pi(M,-1) \ot {\det}) = 0$, because the naive Jacquet module of locally analytic principal series representations never vanishes, cf. \cite[5.1.7]{KisinStrauch06}. We have thus reduced the problem to the case when $k=1$.

\vskip8pt

{\it Step 4.} Our goal is now to show that the naive Jacquet module of $\Pi(M,-1) \ot {\det}$ vanishes. We are going to use some of the main results of \cite{DospinescuLeBras}. Let $D$ be the quaternion division algebra over $\Qp$ and $\cO_D$ its ring of integers. Denote by $\breve{\bbZ}_p$ the completion of the maximal unramified extension of $\Zp$. Let $\cM_0$ be the Rapoport-Zink space which is the formal moduli scheme of special formal $\cO_D$-modules $X$ on schemes $S$ over $\breve{\bbZ}_p$ on which $p$ is locally nilpotent, together with a quasi-isogeny $\bbX \times_{\overline{\bbF}_p} (S \bmod p) \ra X \times_S (S \bmod p)$, where $\bbX$ is a fixed special formal $\cO_D$-module over $\overline{\bbF}_p$. 

\vskip8pt

Denote by $\breve{\cM}_0$ the rigid analytic space over $\breve{\bbQ}_p = \breve{\bbZ}_p[\frac{1}{p}]$ associated to $\cM_0$. It is a disjoint union of countably-infinite many copies of the $p$-adic upper-half plane over $\breve{\bbQ}_p$ (indexed by the height of the quasi-isogeny on the mod-$p$ fiber). This rigid analytic space is equipped with a Weil descent datum which becomes effective after taking the quotient by the action of $p\triv := \diag(p,p)$. Let $\Sigma_0$ be the corresponding model over $\Qp$ of $\breve{\cM}_0/(p\triv)^\Z$. We then consider the Drinfeld covering space $\breve{\cM}_n$ over $\breve{\cM}_0$ which parametrizes primitive $p^n$-torsion points on the universal deformation of $\bbX$ over $\breve{\cM}_0$. It is a Galois covering of $\breve{\cM}_0$ with Galois group $H_n := \cO_D^\x/(1+p^n\cO_D)$. Finally, the Weil descent datum on $\breve{\cM}_0$ lifts to a Weil descent datum on $\breve{\cM}_n$ and it becomes effective on $\breve{\cM}_n/(p\triv)^\Z$. Let $\Sigma_n$ be the corresponding model over $\Qp$ of this quotient.  

\vskip8pt

As before, we have the smooth supercuspidal representation $\pi = {\rm LL}(M)$, and we let $\psi = {\rm JL}(\pi)$ be the representation of $D^\x$ which corresponds to $\pi$ by the Jacquet-Langlands correspondence. We set $\Sigma_{n,E} = \Sigma_n \times_{\Sp(\Qp)} \Sp(E)$. After possibly enlarging $E$ we may assume that all geometrically connected components of $\Sigma_n$ are defined over $E$. In particular, $\Sigma_{0,E}$ is isomorphic to the disjoint union of two copies of the $p$-adic upper-half plane. By \cite[Thm. 1.2]{DospinescuLeBras}, the representation $\Pi(M,-1) \ot {\det}$ is isomorphic to $\Big(H^0(\Sigma_{n,E},\cO)^\psi\Big)'_b$ for some $n > 0$, where $(-)'_b$ denotes the continuous dual space with the strong topology. Since $\Sigma_{n,E}$ is a Stein space, $H^0(\Sigma_{n,E},\cO)^\psi$ is a nuclear Fr\'echet space and thus reflexive. Therefore, it suffices to show that the subspace of $U$-invariant vectors of $H^0(\Sigma_{n,E},\cO)^\psi$ is zero. For a suitable global coordinate $z$ on a connected component of $\Sigma_{0,E}$, the right action of $U$ is given by 

\[z.\left(\begin{array}{cc}
1 & u \\
0 & 1
\end{array}\right) = z+u \;.\]

\vskip8pt

This shows that if $f \in H^0(\Sigma_{0,E},\cO)$ is $U$-invariant, it is locally constant, i.e., constant on each connected component of $\Sigma_{0,E}$. Now let $f \in \cO(\Sigma_{n,E})$ be $U$-invariant. The polynomial $\prod_{h \in H_n} (T - h(f))$ has coefficients in $H^0(\Sigma_{0,E},\cO)$, because $H^0(\Sigma_{0,E},\cO) \ra H^0(\Sigma_{n,E},\cO)$ is a finite ring extension which is Galois with group $H_n$.  These coefficients are also $U$-invariant and hence constant on each connected components of $\Sigma_{0,E}$. On each connected component $C$ of $\Sigma_{n,E}$, the function $f|_C$ is thus a root of a polynomial with coefficients in $E$, and hence itself constant (as $E$ is integrally closed in $H^0(C,\cO)$). It follows that $H^0(\Sigma_{n,E},\cO)^U$ consists of locally constant functions.  

\vskip8pt

The action of $D^\x$ on the space of locally constant functions on $\Sigma_{n,E}$ factors through the norm, cf. \cite[sec. 4.3]{Carayol90}. But $f$ lies in an isotypic component $H^0(\Sigma_{n,E},\cO)^\psi$ for an irreducible representation $\psi$ which does \emph{not} factor through the norm, then $f$ must be zero. This completes the proof in case (2a).

\vskip8pt

(2b) Let $\rho$ again be non-trianguline of type (2), de Rham with Hodge-Tate weights $0$ and $k>0$, and suppose $W := \WD(D_\pst(\rho))$ is not absolutely irreducible. We first note that $W$ must be irreducible, because otherwise a rank 1 subobject of $D_\pst(\rho)$, together with the induced filtration, gives rise to a saturated rank 1 subobject of $D_\rig(\rho)$, which is a triangulation since $D_\rig(\rho)$ is of rank 2, cf. \cite[1.18 (2)]{Nakamura_Classification}, \cite[4.1]{BreuilSchneider}.

\vskip8pt

Set $E_2 = \End_{E[\sW_\Qp]}(W)$, where $\sW_\Qp$ is the Weil group of $\Qp$. Because $W$ is irreducible, the ring $E_2$ is a division algebra, and because $W$ is not absolutely irreducible, $E_2 \neq E$. Because $2 = \dim_E(W) = \dim_E(E_2) \cdot \dim_{E_2}(W)$, we find $[E_2:E] = 2$ and the action of $\sW_\Qp$ on $W$ is given by a (locally constant) character $\tau: \sW_\Qp \ra E_2^\x$. In the following we write $(-)_{E_2}  = (-) \ot_E E_2$. Let $\sigma \in \Gal(E_2/E)$ be the non-trivial Galois automorphism. Then $W_{E_2} \simeq \tau \oplus \sigma \circ \tau$. Since $W$ is irreducible, $\tau$ does {\it not} factor through $\Ex \hra E_2^\x$, and it follows that 

\begin{numequation}\label{character-conditions}
\tau (\sigma \circ \tau)^{-1} \notin \{\triv, |\cdot|^{\pm 1}\} \;.
\end{numequation}

\vskip8pt

Here $|\cdot|: \sW_\Qp \ra \sW_\Qp^{\rm ab} \simeq \Qpx \stackrel{|\cdot|_p}{\lra} p^\Z$, where the $p$-adic absolute value is normalized by $|p|_p = p^{-1}$, as usual. We have 

\[\begin{array}{rcl}
\Pi(\rho_{E_2})^\lalg & \simeq & \Sym^{k-1}(E_2^2) \ot_{E_2} \LL\Big(\WD(D_\pst(\rho_{E_2}))\Big) \\
&&\\
& \simeq & \Sym^{k-1}(E_2^2) \ot_{E_2} \ind^G_B\Big(\tau \ot (\sigma \circ \tau) |\cdot|^{-1}\Big) \;,
\end{array}\]

\vskip8pt

where the first isomorphism holds by \cite[0.20, 0.21]{Colmez10} and the second isomorphism by \cite[subsection 11 in VI.6]{Colmez10}. Because of \ref{character-conditions} the representation $\ind^G_B\Big(\tau \ot (\sigma \circ \tau) |\cdot|^{-1}\Big)$ is irreducible \cite[9.6]{BushnellHenniart_GL2}. If $\Pi^\lalg$ were a representation of principal series type, it would be of the form $\Sym^{k-1}(E^2) \ot_E \ind^G_B\Big(\chi_1 \ot \chi_2|\cdot|^{-1}\Big)$, with $E$-valued characters $\chi_1$, $\chi_2$ satisfying $\chi_1\chi_2^{-1} \neq |\cdot|^{\pm 1}$. The base change to $E_2$ would be $\Sym^{k-1}(E_2^2) \ot_{E_2} \ind^G_B\Big(\chi_{1,E_2} \ot \chi_{2,E_2}|\cdot|^{-1}\Big)$. The argument in {\it Step 2} of case (2a) then implies 

\[\ind^G_B\Big(\chi_{1,E_2} \ot \chi_{2,E_2}|\cdot|^{-1}\Big) \simeq \ind^G_B\Big(\tau \ot (\sigma \circ \tau) |\cdot|^{-1}\Big)  \;.\]

\vskip8pt

By \cite[9.10]{BushnellHenniart_GL2} this implies 

\[\chi_{1,E_2} \ot \chi_{2,E_2}|\cdot|^{-1} = \tau \ot (\sigma \circ \tau) |\cdot|^{-1} \;\; \mbox{ or } \;\; \chi_{1,E_2} \ot \chi_{2,E_2}|\cdot|^{-1} = (\sigma \circ \tau) \ot \tau |\cdot|^{-1} \;.\]

\vskip8pt

Either equality would show that $\tau$ is $E$-valued, contradicting the preceding observation that $\tau$ is not $E$-valued. It follows that $\Pi^\lalg$ is not of principal series type.

\vskip8pt

Now we consider $\Pi^\la/\Pi^\lalg = \Pi^\la/\Pi^\SLlalg$, which is irreducible by \ref{Colmez-poids} (2). If this representation would be of principal series type, it would have to be an irreducible principal series, by \ref{dim-of-l-a-vectors} and \ref{JH-princseries}. Its base change to $E_2$ would be an irreducible principal series, cf. the proof of \ref{JH-princseries}. Since $\rho_{E_2}$ is trianguline with distinct Hodge-Tate weights, \cite[0.4 (ii), bullet points 3 and 4]{Colmez_poids}\footnote{In the statement of the fourth bullet point there is the number '4' missing.} shows that $\Pi(\rho_{E_2})^\la$ has length three or four. Hence $\Big[\Pi^\la/\Pi^\lalg\Big] \ot_E E_2$ is of length at least two. Therefore, $\Pi^\la/\Pi^\lalg$ cannot be of principal series type.
\end{proof}

\printbibliography

\end{document}